\documentclass[11pt,leqno,a4paper,twoside]{article} 
\usepackage[T1]{fontenc}
\usepackage[applemac]{inputenc} 
\usepackage[english]{babel}
\usepackage{amsmath,amssymb,amsthm}
\usepackage{amsfonts}
\theoremstyle{plain}  
\newtheorem{theorem}{Theorem}[section] 

\newtheorem{lemma}[theorem]{Lemma}

\theoremstyle{remark}
\newtheorem*{remark}{Remark}
\newtheorem*{remarks}{Remarks}
\usepackage[dvipsnames]{xcolor}

\usepackage{euscript} 
\usepackage{mathrsfs}
\usepackage{fancyhdr}
\usepackage{comment} 
\newcommand{\A}{{\EuScript{A}}}
\newcommand{\B}{{\EuScript{B}}} 
\newcommand{\D}{{\EuScript{D}}}
\newcommand{\E}{{\EuScript{E}}} 
 
\newcommand{\M}{{\EuScript{M}}}
\def\N{{\mathbb N}}
\def\R{{\mathbb R}}
\newcommand{\sset}{\smallsetminus}
\newcommand{\T}{{\EuScript{T}}}
\newcommand{\U}{{\EuScript{U}}}
\newcommand{\V}{{\EuScript{V}}}

\def\1{{\bf 1}}
\def\bfc{{\boldsymbol c}}
\def\e{{\rm e}}
\def\d{\,{\rm d}}
\def\dd{{\rm d}}
\def\dsp{\displaystyle}
\newcommand{\norm}[1]{\left\|#1\right\|}
\newcommand{\Vbs}[1]{\bigg|#1\bigg|}
\newcommand{\fl}[1]{\left\lfloor#1\right\rfloor}
\newcommand{\md}[2]{\equiv#1\,({\rm mod\,}#2)}
\numberwithin{equation}{section}

               \countdef\temps=170
                \temps=\time
                \countdef\nminutes=171{\nminutes = \time}
                \countdef\nheure=172
                \def\heure{\begingroup
                   \temps = \time \divide\temps by 60
                   \nheure = \temps
                   \nminutes = \time
                   \multiply\temps by 60
                   \advance\nminutes by -\temps
                   \ifnum\nminutes<10 \toks1 = {0}%
                   \else\toks1 = {}%
                   \fi
                   \number\nheure h\the\toks1 \number\nminutes
                \endgroup}%

\def\dater{\vglue-10mm\rightline{(\the\day/\the\month/\the\year)}}
                \def\dateheure{(version \the\day/\the\month/\the\year,\ \heure)}

\title{\bf Small Gl sums and applications}
 
\author{R\'egis de la Bret\`{e}che,  Marc Munsch \& Grald Tenenbaum\\
\smallskip
\centerline{\small\dateheure}\\ 
}
  
\newcommand{\Addresses}{{% additional braces for segregating \footnotesize
  \bigskip
  \footnotesize
\textsc{R\'{e}gis de la Bret\`{e}che, IMJ-PRG, Universit\'e de Paris, Sorbonne Universit\'e,
   \goodbreak   CNRS UMR 7586, Case 7012, F-75013 Paris, France}\par\nopagebreak
  \textit{E-mail address:} \texttt{regis.de-la-breteche@imj-prg.fr}
\par 
  \medskip
\textsc{Marc Munsch,  Institut f\"{u}r Analysis und Zahlentheorie, 
8010 Graz, \goodbreak Steyrergasse 30, Graz, Austria}\par\nopagebreak
  \textit{E-mail address:} \texttt{munsch@math.tugraz.at}
 \par  
\medskip
\textsc{Grald Tenenbaum,  Institut lie Cartan de Lorraine, \goodbreak
Universit de Lorraine, BP 70239,
54506 VandÏuvre-ls-Nancy Cedex,
France
}\par\nopagebreak
  \textit{E-mail address:} \texttt{gerald.tenenbaum@univ-lorraine.fr}
}}

\date{~}
\begin{document}
\bibliographystyle{abbrv}
\maketitle

\footnotetext{
2010 Mathematics Subject Classification. 
Primary: 11L40, 11N37. Secondary : 05D05, 11F27. \\
Key words and phrases. GCD sums, multiplicative energy, character sums, Burgess' bound, theta functions, mollifiers.
}

\begin{abstract}
In recent years, maximizing Gl sums regained interest due to a firm link with large values of $L$-functions. In the present paper, we initiate\footnote{After a preprint by the second author \cite{energy} was released, the authors worked together  and obtained several improvements as well as other results which are now contained in this  version.} an investigation of small sums of Gl type, with respect to the $L^1$-norm.  We also consider the intertwined question of minimizing weighted versions of the usual multiplicative energy. We apply our estimates to: (i) a logarithmic refinement of Burgess' bound on character sums, improving previous results of Kerr, Shparlinski and Yau; (ii) an improvement on earlier lower bounds by Louboutin and the second author for the number of non vanishing theta functions associated to Dirichlet characters; and (iii) new lower bounds for low moments of character sums.
\end{abstract}

\bibliographystyle{plain}
\maketitle

\section{Introduction and statements of results}\label{Intro}
\subsection{Gl sums}\label{gcd}
\hspace{\parindent} 
 Given a subset $\M$ of the set $\N^*$ of positive integers and an exponent $\alpha\in]0,1]$, one traditionally defines the G\'al sum
$$S_\alpha( \M):=\sum_{m,n\in\M} \frac{(m,n)^{\alpha}}{ [m, n]^\alpha}=\sum_{m,n\in\M} \frac{(m,n)^{2\alpha}}{ (mn)^\alpha},$$ where $(m,n)$ denotes the greatest common divisor of $m$ and~$n$ and $[m,n]$ stands for their smallest common multiple.\par 
Bounding these sums had originally interesting applications in metric Diophantine approximation---see \cite{DuffinHarman, metric}. Recently, further study was carried out due to the connection with large values of the Riemann zeta function---see for instance \cite{polydisc, hilb,GalRad,sound}. \par 
In \cite{bs1, Gal}, lower bounds for maximal G\'al sums have been used to obtain lower bounds for 
 $$\max_{t\in [0,T]} |\zeta(\tfrac 12 +it)|,\quad\max_{\chi\in X_p^+} |L(\tfrac 12,\chi)|,$$ where 
$X_p^+$ is the set of even characters modulo $p$ and $L(s,\chi)$ is the $L$-series associated to a Dirichlet character $\chi$. In \cite{Gal}, La Bret\`eche and Tenenbaum proved that  
$$\max_{|\M |=N} \frac{S_{1/2}( \M )}{|\M |}=
\exp\Bigg\{\big(2\sqrt{2}+o(1)\big)\sqrt{\frac{\log N \log_3N}{ \log_2 N}}\Bigg\},$$
where, here and throughout, we denote by $\log_k$  the $k$-th  iterated  logarithm.
In this result, the cardinality of $\M $ is fixed while the size of its elements is not.
Moreover, it is also shown in \cite{Gal} that  the same estimate holds for
$$Q(\M ):=\sup_{\substack{\bfc\in \mathbb{C}^N\\ \norm\bfc_2=1}}
\Bigg| \sum_{m,n\in\M }c_m\overline{c_n} 
\frac{(m,n) }{ \sqrt{mn}}\Bigg| $$
where $\norm\bfc_p$ denotes the $p$-norm of the $N$-tuple $\bfc\in \mathbb{C}^N$.
\par 
In this work, we consider, for $\bfc=(c_n)_{1\leqslant n\leqslant N}\in\R^N$, the quantities 
\begin{equation*}
\begin{aligned}
&\V(\bfc;N):=\sum_{m , n\leqslant N} \frac{(m,n)}{m+n}c_mc_n,\quad\V_N:=N\inf_{\substack{ \bfc\in (\mathbb{R}^{+})^N \\ \norm\bfc_1=1 }} \V(\bfc;N),
\\
&\T(\bfc;N):=\sum_{m , n\leqslant N} \frac{(m,n)}{\sqrt{mn}}c_mc_n,\quad \T_N:=N\inf_{\substack{ \bfc\in (\mathbb{R}^{+})^N \\ \norm\bfc_1=1 }}\T(\bfc;N),
\end{aligned}
\end{equation*}
with the aim of asymptotically evaluating $\V_N$ and $\T_N$ as precisely as possible. \par 
 Our main result is as follows. For $u\in]0,1]$, we put
\begin{equation*}
\begin{aligned}
&y(u):=\tfrac12(-1+\sqrt{1+4u}),\ f(u):=2u\log \{1+y(u)\}-y(u)^2,\\
& Q(u):=u\log u-u+1,
\end{aligned}
\end{equation*}
and let $\beta\approx0.48155$ denote the solution to the equation $f(\beta)=Q(\beta)$. We have  $y(\beta)\approx0.35530$ and $\eta:=f(\beta)\approx0.16656<1/6$.  

\begin{theorem}
\label{gcdsumth}
 For $N\geqslant 3$, we have  \begin{equation}
\label{encVT}
(\log N)^{ \eta}\ll  \V_N\leqslant \tfrac12\T _N\ll   (\log N)^{ \eta}(\log_2N)^3.
\end{equation}
\end{theorem} 
This minimization question arises naturally in three different contexts. The first concerns  Burgess' famous  bound on multiplicative character sums, for which we derive a logarithmic improvement---see Section \ref{secburgess}.  The second  is related with non-vanishing of theta functions: in this case, it turns out that a related minimization problem yields better results---cf. Sections \ref{multiplisec} and \ref{nvth}. As an application of this latter estimate, we obtain lower bounds for low moments of character sums, stated in Section \ref{lbch}. This minimization problem might also have applications in metric Diophantine approximation. 
 \subsection{Multiplicative energy}\label{multiplisec}
Given two integer sets $\A ,\B \subset \left[1,N\right]$, the multiplicative energy, as defined for instance in \cite{Gowers4,Taogroup,TaoVu}, is the quantity
 $$ E(\A ,\B ):=\sum_{n\geqslant 1}\Bigg(\sum_{\substack{a\in\A,\,b\in\B\\ab=n}}1\Bigg)^2.$$  This notion turns out  to be of great importance in additive combinatorics. Under the additional restriction $\bfc\in \left\{0,1\right\}^N$, the weights $\bfc$ considered in Section \ref{gcd} can be viewed as indicators $\bfc=\1_{\B }$ of  integer sets~$\B \subset \left[1,N\right]$. In such a setting,  evaluating $\V_N$ amounts to minimizing $S(\B )/|\B |^2$ where 
\begin{equation}\label{defS} S(\B ):=\sum_{m,n \in \B  }\frac{(m,n)}{m+n}.\end{equation} It is not hard to see that this sum is intimately connected to 
 $E([1,N],\B ).$\par \goodbreak
    With  applications in mind, we need to bound the multiplicative energy in a symmetric situation, namely $E(\B ,\B )$. To be consistent with our previous approach and provide some flexibility, we define the weighted version of the multiplicative energy
\begin{equation}\label{weightedenergy} 
\E(\bfc;N):=\sum_{1\leqslant n\leqslant N^2}\Bigg(\sum_{\substack{d\leqslant N,\,t\leqslant N\\ dt=n}}c_dc_t\Bigg)^2,  \end{equation} 
and in turn
\begin{equation}\label{defm}  \E_N:=\inf_{\substack{\bfc\in (\mathbb{R}^{+})^N,\, \norm\bfc_1=1}} N^2 \E(\bfc;N), 
 \end{equation}
so that, if one restricts $\bfc=\1_{\B }$ to the indicator of sets $\B\subset [1,N] $, the corresponding problem amounts to minimizing $ {N^2 E(\B ,\B )}/{\vert \B \vert^4}$. \par 
Using  techniques similar to those employed in the proof of Theorem~\ref{gcdsumth}, we obtain
the following result, where we write 
\begin{equation}
\label{delta}
\delta:=Q(1/\log 2)=1-(1+\log_2 2)/\log 2\approx 0.08607
\end{equation}
 for the exponent appearing in the famous multiplication table problem of Erd\H{o}s \cite{Tenintervalle,Ford}.
\begin{theorem}\label{multh}
For $N\geqslant 3$, we have   
\begin{equation}
\label{estEN}
(\log N)^\delta(\log_2N)^{3/2}\ll \E_N\ll (\log N)^{ \delta}\big(\log_2N\big)^{6}.
\end{equation}
\end{theorem}  
  
  \subsection{Improvement of Burgess' bound}\label{secburgess}
\hspace{\parindent} 
Let us consider
$
S(M,N;\chi) := \sum_{M<n \leqslant M+N} \chi(n),$
where $\chi$ is a Dirichlet character to the modulus $p$. The classical bound of P\'olya and Vinogradov reads 
\begin{equation}\label{PV}
S(M,N;\chi) \ll \sqrt{p} \log p
\end{equation}
provided $\chi$ is non principal. This is non trivial  for~$N > p^{1/2+\varepsilon}$.
A major breakthrough has been obtained by Burgess \cite{Burgesssaving,Burgesssaving2}, providing a saving for intervals of length $N \geqslant p^{1/4+\varepsilon}$. Precisely, for prime  $p$, non principal Dirichlet character modulo $p$, and integer $r\geqslant 1$, Burgess proved the following inequality
\begin{equation} \label{eq:Burgess}
S(M,N;\chi) \ll N^{1-1/r} p^{(r+1)/4r^2} \log p,
\end{equation} where the constant depends only on $r$. Even though much stronger results are expected, this bound remains  the sharpest unconditional estimate to date, as far as the values of the exponents are concerned.
 \par 
However, some logarithmic refinements have been obtained unconditionally---see \cite[chapter 14]{IK} following ideas from \cite{FrI}. The sharpest known is due to Kerr, Shparlinski and Yau, who proved
\begin{equation}\label{IgorKam}  S(M,N;\chi)  \ll N^{1-1/r} p^{(r+1)/4r^2} (\log p)^{1/4r}\qquad (r\geqslant 2). \end{equation} These improvements rely on an averaging argument which leads to counting the number of solutions of certain congruences modulo $p$. Initially, the averaging was carried out over the full interval, while, in \cite{Burgessrefine}, the authors restricted it to integers free of small prime factors. Theorem~\ref{gcdsumth} allows a similar argument with a set of higher density. We obtain the following result.
\begin{theorem}
\label{Burgess}
Let $r\geqslant 2$ be fixed. For all prime $p$, all integers $M$, $N$ such that
$$1\leqslant N\leqslant p^{1/2+1/4r},$$
and any non principal Dirichlet character $\chi$  to the modulus $p$, we have
\begin{align*}
\vert S(M,N;\chi)\vert  &\ll N^{1-1/r}p^{(r+1)/4r^2} \max_{1\leqslant x\leqslant p} \T _x^{1/2r}
\\&
\ll N^{1-1/r}p^{(r+1)/4r^2} (\log p)^{\eta/2r}(\log_2p)^{3/2r}
\end{align*}  where $\eta\approx 0.16656$ is the exponent appearing in Theorem~\ref{gcdsumth}.
\end{theorem} 

\subsection{Non vanishing of theta functions}\label{nvth}
  \hspace{\parindent}  The distribution of values of $L$-functions is a deep question in number theory, with  significant consequences for related  arithmetic, algebraic and geometric objects. The main reason comes from the fact that these values, and particularly the central ones, hold much fundamental arithmetical information, as illustrated by the famous conjecture of Birch and Swinnerton-Dyer \cite{BSD1,BSD2}. It is widely believed that these central values should not vanish unless an underlying arithmetic reason forces~it.\par 
   Consider the Dirichlet $L$-functions associated to multiplicative characters $$L(s,\chi):=\sum_{n \geqslant 1} \frac{\chi(n)}{n^s} \hspace{5mm} (\Re e(s) > 1).$$ In this case, no algebraic property imposes the vanishing of $L(\tfrac{1}{2},\chi)$. Therefore it is  expected that this number is non-zero, as firstly conjectured by Chowla \cite{Chowla} for quadratic characters. Using the method of mollification, it was first proved by Balasubramanian and Murty \cite{BalaMurty} that a positive proportion of characters verify $L(\tfrac{1}{2},\chi)\neq0$.  One of the main analytic tools is the study of moments. Various authors have obtained mean value results on the values $L(\tfrac{1}{2},\chi)$.
 
As initiated in previous works \cite{Debrecen,LM,thetalow,thetaupp}, we aim at similar results for moments 
of theta functions $\vartheta (x;\chi )$ associated with Dirichlet characters, as defined by
\begin{equation}
\label{deftheta}
 \vartheta (x;\chi)
=\sum_{n\geqslant 1} \chi (n){\rm e}^{-\pi n^2x/p}\quad(\chi\in X_p^+),
\end{equation}
where 
$X_p^+$ denotes the subgroup of order $\tfrac{1}{2}(p-1)$ of even Dirichlet characters modulo~$p$.
  It was conjectured by Louboutin \cite{efficient} that $\vartheta(1;\chi) \neq 0$ for every non-trivial character modulo a prime\footnote{Pascal Molin informed us that his computations show that $\vartheta(1;\chi) \neq 0$ for $p \leqslant 10^6$.}---see \cite{Zagier} for a case of vanishing in the composite case. Using  evaluations of the second and fourth moments of $\vartheta(1;\chi)$, Louboutin and the second author \cite{LM} showed that $\vartheta (1;\chi)\neq 0$ for at least $p/\log p$ even characters modulo $p$---for odd characters,  
a similar result has already been proved by Louboutin in \cite{CRAS}. Constructing  mollifiers of a different kind than those devised for $L$-functions, we obtain the following improvement. 

\begin{theorem}\label{mainth}
Let $x>0$ be fixed. For suitable constant $c$, sufficiently large prime~$p$ and $q:=\lfloor\sqrt{p/3}\rfloor$, there exists at least $$\gg  \frac{p}{\E_q} \gg \frac{p}{(\log p)^{ \delta}(\log_2p)^6}$$ even characters $\chi$ such that $\vartheta(x;\chi) \neq 0$, where $\delta \approx 0.08607$ is as defined in \eqref{delta}.
\end{theorem}  

At this point, some methodological comments may be useful. In order to prove that $\vartheta(x;\chi)\neq 0$ for many  $\chi\in X_p^+$, the usual approach rests on evaluating the asymptotic behavior of the moments 
$$S_{2k}^+(p) :=\sum_{\chi\in X_p^+}\vert\vartheta(x;\chi)\vert^{2k}  \hspace{10mm} (k > 0).$$ 

 Lower bounds providing the expected order for the moments are obtained in \cite{thetalow}. Conditionally on GRH, nearly optimal upper bounds are established in \cite{thetaupp}. 
\par 
This strategy is related to recent results by Harper, Nikeghbali and Radziwill in \cite{HarperMaks}, where the asymptotic behavior of moments of a Steinhaus random multiplicative function\footnote{That is a random multiplicative function  whose values at primes are uniformly distributed on the complex unit circle.} is determined. Indeed a Steinhaus function  may be regarded as a random model for $\vartheta(x;\chi)$: on the one hand, the effect of the rapidly decaying factor $\e^{-\pi n^2/p}$ is essentially equivalent to restricting the summation in \eqref{deftheta} to a finite range $n \leqslant N_p$ with $N_p \approx \sqrt{p}$, and, on the other hand, the  behavior of $\chi(n)$ for $n\ll \sqrt{p}$ is similar on average to that of the Steinhaus random function. \par 
As noticed   in \cite{HarperMaks}, an asymptotic formula for the first absolute moment $S_{1}^+(p)$
 would probably imply the existence of a positive proportion of characters such that $\vartheta(x;\chi) \neq 0$. However,   Harper recently announced, in both the random and the deterministic case, that, unexpectedly, the first moment  exhibits  more than square-root cancellation: for $2<p$, $1\leqslant N\leqslant p/3$, we have
\begin{equation}
\label{squareroot}
\frac{1}{p-2} \sum_{\chi \neq \chi_0} \Vbs{\sum_{n\leqslant N} \chi(n)} \ll \frac{\sqrt{N}}{(\log_23\nu)^{1/4}}\end{equation} where $\nu:=\min\left\{N,p/N\right\}$. This estimate shows that the above mentioned strategy would, in any case, fail to detect a positive proportion of ``good'' characters. \par 
 We  undertook another approach, specifically introducing  suitable Dirichlet polynomials as mollifiers, and thereby reducing the problem to the minimizations considered in Section \ref{Intro}.
\par 
In the next section, we also state a lower bound for the first moment~\eqref{squareroot}.

\subsection{Lower bounds for low moments of character sums}\label{lbch}
\hspace{\parindent} We write $f(x)\asymp g(x)$ to mean that $g(x) \ll f(x) \ll g(x)$, in other words, that the inequalities 
$c_0g(x)\leqslant f(x)\leqslant c_1g(x)$ hold for some constants $c_0,\,c_1$ and all $x$ in some specified range. \par 
Let us consider $S(N;\chi)= \sum_{n\leqslant N} \chi(n)$, where $\chi$ is a Dirichlet character modulo a prime $p$. As recalled in the previous section,  Harper, using probabilistic techniques, recently proved Helson's surprising conjecture that the first moment of Steinhaus random multiplicative functions is  $o(\sqrt{N})$, and indeed he showed \cite{Harpelow} that
$$ \mathbb{E}\bigg\vert \sum_{n\leqslant N} f(n)\bigg\vert \asymp \frac{\sqrt{N}}{(\log_2 N)^{1/4}}\cdot$$
 \par 
As for the deterministic case, this approach yields the upper bound \eqref{squareroot} for the first moment of character sums.
Obtaining sharp lower bounds by the probabilistic methods developed in \cite{Harpelow} seems harder.\footnote{Private communication with Adam Harper.} Using Theorem \ref{multh}, we establish the following lower bound for some $L^r$-norms of character sums. \par 
Recall the definition of $\delta$ in \eqref{delta} and note that 
$\delta/2 \approx 0.04304$.
\begin{theorem}\label{firstmomentcarac}
Let $r\in]0,4/3[$ be fixed. For sufficiently large $p$, all $N\in\big[1,\tfrac12p\big[$ and $\nu:=\min(N,p/N)$, we have  
\begin{equation}%\label{lowcarac} 
\frac{1}{p-2} \sum_{\chi \neq \chi_0} \left\vert S(N;\chi)\right\vert^r \gg \frac{N^{r/2}}{\E_\nu^{1-r/2}}.
 \end{equation}In particular,  
$$
\frac{1}{p-2} \sum_{\chi \neq \chi_0} \left\vert S(N;\chi)\right\vert \gg \sqrt{\frac{N}{\E_\nu}} \gg \frac{\sqrt{N}}{(\log \nu)^{ \delta /2}(\log_2\nu)^3}.$$
\end{theorem}
\begin{remarks}
(i)
This result can be easily generalized to
composite moduli, but for the sake of simplicity and consistence, we restricted to the case of prime moduli.
\par 
(ii) A preliminary version of Theorem \ref{firstmomentcarac} was stated under the assumption $N\leqslant \sqrt{p}$.  Ping Xi noted that Plya's formula \eqref{Polya} readily enables extending the validity to the full range $N\leqslant \tfrac12p$. Moreover he also noticed that Theorem \ref{firstmomentcarac} still holds if $S(N,\chi)$ is replaced by $$\sum_{n\leqslant N}\varepsilon_n\chi(n)$$ with $|\varepsilon_n|=1$ $(1\leqslant n\leqslant N)$.
\end{remarks}

\par 
The same method may be also applied to lower bounds for 
$$ \frac{1}{T}\int_0^T \Vbs{\sum_{n\leqslant N}n^{it}}^r {\rm d} t.$$ The study of the limit of this mean-value as $T\to\infty$ was initiated by Helson \cite{H06} and further investigated by Bondarenko and Seip in~\cite{BS16}. For any $r \leqslant 1 $, they obtained a lower bound of size $\sqrt{N} (\log N)^{-0.07672}$ and could improve the exponent to $-0.05616$ for $r=1$,  using a  method different from ours. Indeed, their approach relies on  \cite[Lemma $3$]{BS16}, which is not available for character sums. The following result illustrates the relevance of Theorem \ref{multh} in this context.

\begin{theorem}\label{firstmomentpolyzeta}
Let $r\in]0,4/3[$. For $T\geqslant 1$, $1\leqslant N\leqslant \sqrt{T}$,  we have  
\begin{equation} 
 \frac{1}{T}\int_0^T  \Vbs{\sum_{n\leqslant N}n^{it}}^r {\rm d} t\gg \frac{N^{r/2}}{\E_N^{1-r/2}}.
 \end{equation}In particular, 
$$
\lim_{T\to+\infty} \frac{1}{T}\int_0^T  \Vbs{\sum_{n\leqslant N}n^{it}}  \d t \gg \sqrt{\frac{N}{\E_N}} \gg \frac{\sqrt{N}}{(\log N)^{ \delta /2}(\log_2N)^3}.$$ 
\end{theorem}
As the proof of  Theorem \ref{firstmomentpolyzeta} is similar to that of Theorem~\ref{firstmomentcarac}, we omit the details.
\begin{remark}
Studying the proofs of \cite{BS16}, one may ask whether their method also provides the same exponent $\delta/2$.\footnote{This was answered positively, without details, in May 2017 by ``Lucia''  on mathoverflow: https://mathoverflow.net/questions/129264/short-character-sums-averaged-on-the-character. (The stated exponent $\delta/4$ apparently stems from an oversight.) As noted by ``Lucia'', the method of \cite{BS16} relies on some input from analysis (\cite[lemma 3]{BS16}) which enables restricting the summation to a level set for the $\Omega$ function, defined in Section \ref{pfth1.1}. However, as will be seen from our proof, it is rather doubtful that this approach may yield the right exponent.} \end{remark}

\section{Proof of Theorem \ref{gcdsumth}}\label{pfth1.1}
\subsection{A lemma on the distribution of prime factors}
Let $\Omega(n)$ (resp. $\omega(n)$) denote the total number of prime factors, counted with (resp. without) multiplicity, of an integer $n$. For $k\geqslant 1$, $x\geqslant 1$, let $N_k(x)$ denote the number of those integers $n\leqslant x$ such that $\Omega(n)=k$. Defining $\Omega(n,t):=\sum_{p^\nu\|n,\,p\leqslant t}\nu$, we also consider the number $F_k(x;C)$ of those integers $n\leqslant x$ counted by $N_k(x)$ and such that
\begin{equation}
\label{loc}
\Omega(n,t)\leqslant \kappa\log_23t+C\quad(1\leqslant t\leqslant x)
\end{equation}
with $\kappa:=k/\log_2x$.
\begin{lemma} Let $\kappa_0\in]0,2[$. For $0\leqslant \kappa\leqslant \kappa_0$ and suitable $C=C(\kappa_0)$, we have
\begin{equation}
\label{minFk}
F_k(x;C)\asymp\frac{N_k(x)}{k}\quad(x\geqslant 3).
\end{equation}
\end{lemma}
\begin{proof} This  is stated in \cite[Th.~1]{Fo07} in the case prime factors are counted without multiplicity. Since the proof of the lower bound given in this work is actually set up for squarefree integers, the lower bound included in \eqref{minFk} trivially follows. 
\par 
For the upper bound, we still appeal to \cite[Th.~1]{Fo07}, with some adjustment. 
 Consider an integer $n$ counted by $F_k(x;C)$. Then $n$ may be represented uniquely as $n=ab$, where $a$ is squarefull, $b$ is squarefree, $(a,b)=1$, and $\omega(b)=\ell\leqslant k$. Since $a$ is squarefull, a standard argument enables us to discard those $n$ such that $a>(\log x)^3$, say.
Ford's estimate  quoted above implies that the number of squarefree $b$ not exceeding $x/a$, verifying \eqref{loc}  and such that $\omega(b)=\ell$ is
$$\ll \frac{(k-\ell+1)x(\log_2x)^{\ell-1} }{ a\ell !  \log x}\cdot$$
Letting $a$ denoting generically a squarefull integer, this yields
$$F_k(x;C)\ll \frac{x}{\log x}\sum_{\ell\leqslant k}\frac{(k-\ell+1)(\log_2x)^{\ell-1}G_\ell}{ \ell !}, $$
with 
\begin{align*}
G_\ell&:=\sum_{\substack{a\geqslant 1\\ \Omega(a)=k-\ell}}\frac{1}{a}\leqslant \sum_{j\geqslant 0}\frac1{2^{j}}\sum_{\substack{a\geqslant 1,\, 2\,\nmid\,a\\ \Omega(a)=k-\ell-j}}\frac{1}{a}\\
& \leqslant \sum_{0\leqslant j\leqslant k-\ell}\frac1{2^{j}}\sum_{a\geqslant 1,\, 2\,\nmid\,a}\frac{(5/2)^{\Omega(a)-k+\ell+j}}{a}\ll\frac1{2^{k-\ell}}\cdot
\end{align*}
Thus
\begin{align*}F_k(x;C)&\ll  N_k(x) \sum_{\ell\leqslant k} \frac{(k-\ell+1)(\log_2x)^{\ell-k}(k-1)!}{ \ell !2^{k-\ell}} 
\cr&\ll \frac{N_k(x)}{k}\sum_{\ell\leqslant k}(k-\ell+1) \Big(\frac{k}{2\log_2x}\Big)^{k-\ell }  \ll \frac{N_k(x)}{k}\cdot
\end{align*}
\vskip-7mm 
\end{proof}
\subsection{Completion of the proof}
Let us start by proving the lower bound estimate for $\V_N$. The difficulty  partly stems from the hybrid feature of the  definition of $\V(\bfc;N)$, which mingles the  additive and multiplicative structures. \par 
Using an integral representation for $1/(m+n)$ from the theory of the Euler Gamma function (see, e.g., \cite{TW14}, Exercise 154), which also coincides with the computation of the  Fourier transform of $1/\cosh\pi\tau$, we have
\begin{equation}
\label{repintV}
\begin{aligned}
\V(\bfc;N)&=\sum_{d\leqslant N}\frac{\varphi(d)}d\sum_{m,n\leqslant N/d}\frac{c_{md}c_{nd}}{ m+n}\\&=\sum_{d\leqslant N}\frac{\varphi(d)}{ 2d}\int_\R\sum_{m,n\leqslant N/d}\frac{c_{md}c_{nd}}{\sqrt{mn}(m/n)^{i\tau}}\frac{\dd \tau}{\cosh \pi\tau}\\&=\int_\R\sum_{d\leqslant N}\frac{\varphi(d)}{2d}|x_d(\tau)|^2\frac{\dd\tau}{\cosh\pi\tau},
\end{aligned}
\end{equation} 
where we have put
$$x_d(\tau):=\sum_{m\leqslant N/d}\frac{c_{md}}{m^{1/2+i\tau}}\cdot$$
\par 
 Let $y\in]0,1]$. With $s:=\tfrac12+i\tau$, $h:=\1*y^{\Omega}\geqslant (1+y)^\omega$, we may write 
$$S(\tau):=\sum_{d\leqslant N}y^{\Omega(d)}\frac{x_d(\tau)}{d^s}=\sum_{n\leqslant N}\frac{c_nh(n)}{n^s}, $$
whence, by the Cauchy-Schwarz inequality,
$$\Vbs{\sum_{n\leqslant N}\frac{c_nh(n)}{n^s}}^2\leqslant \sum_{d\leqslant N}\frac{\varphi(d)}{d}|x_d(\tau)|^2\sum_{d\leqslant N}\frac{y^{2\Omega(d)}}{\varphi(d)}, $$
and, inserting into \eqref{repintV},
\begin{equation}
\label{minV-1}
\sum_{m,n\leqslant N}\frac{c_mc_n(1+y)^{\omega(m)+\omega(n)}}{m+n}\ll \V(\bfc;N)(\log N)^{y^2}.
\end{equation}
\par Recall the definition of $\beta$ above the statement of Theorem \ref{gcdsumth},
define $$\D^+(\beta):=\{n\leqslant N:\omega(n)\geqslant \beta\log_2N\}, \quad\D^-(\beta):=[1,N]\sset\D^+(\beta)$$ and consider $\bfc\in(\R^+)^N$ such that $\norm\bfc_1=1$. If
\begin{equation}
\label{Dgrand}
\sum_{n\in\D^+(\beta)}c_n\geqslant \tfrac12,
\end{equation}
then, by restricting the summation to $m,n\in\D^+(\beta)$, we get that the left-hand side of  \eqref{minV-1} is $\gg(\log N)^{2\beta\log (1+y)}/N$. The optimal value of the subsequent lower bound for $\V(\bfc;N)$ being achieved for  $y=y(\beta)$, it follows that
\begin{equation}
\label{minVf}
N\V(\bfc;N)\gg(\log N)^{f(\beta)}\cdot
\end{equation}
\par 
If \eqref{Dgrand} does not hold, we define $\sigma_s(n,\beta):=\sum_{d|n,\,d\in\D^-(\beta)}d^s$ and note that
\begin{equation*}
\sum_{d\in\D^-(\beta)}x_d(\tau)=\sum_{n\leqslant N}\frac{c_n}{ n^s}\sigma_s(n,\beta),
\end{equation*}
whence
\begin{equation}
\label{CSD-}
\Vbs{\sum_{n\leqslant N}\frac{c_n}{n^s}\sigma_s(n,\beta)}^2\leqslant \sum_{d\in\D^-(\beta)}\frac{d}{\varphi(d)}\sum_{d\in\D^-(\beta)}\frac{\varphi(d)}{d}|x_d(\tau)|^2.
\end{equation}
By \cite[th. II.6.5]{Tencourse}, the first $d$-sum above is $\asymp N/\{(\log N)^{Q(\beta)}\sqrt{\log_2N}\}$. Inserting \eqref{CSD-} back into \eqref{repintV}, we get
$$\sum_{m,n\leqslant N}c_mc_n\int_\R\sum_{\substack{t|m,\,t\in\D^-(\beta)\\ d|n,\,d\in\D^-(\beta)}}\Big(\frac{dm}{tn}\Big)^{1/2+i\tau}\frac{\dd\tau}{\cosh\pi\tau}\ll\frac{N\V(\bfc;N)}{(\log N)^{Q(\beta)}\sqrt{\log_2N}}.$$
Restrict the outer summation in $m,n$ to $\D^-(\beta)$. Then the inner sum always contain $t=m$, $d=n$. Since the Fourier transform of $1/\cosh\pi\tau$ is positive, we obtain
\begin{equation}
\label{minVQ}
N\V(\bfc;N)\gg(\log N)^{Q(\beta)}\sqrt{\log_2N}.
\end{equation}
Taking \eqref{minVf} into account, we reach the required lower bound.
\par
 Let us now turn our attention to proving the upper bound for $\T_N$ claimed in  \eqref{encVT}.
With $k:=\fl{\beta\log_2N}$, 
we select $n\mapsto c_n$ as the indicator of the set of those integers $n\in[1,N]$ such that $\Omega(n)=k$ and \eqref{loc} holds for  $\kappa=\beta$ and $C=C(\beta)$. Therefore we  have (see, e.g. \cite[th. II.6.5]{Tencourse})
\begin{equation}
\label{c1}
\norm\bfc_1:=\sum_{n\leqslant N}c_n\asymp \frac{N}{ (\log N)^{Q(\beta)}(\log_2N)^{3/2}}.
\end{equation}
Next,
\begin{equation}
\label{xd}
\T(\bfc;N)=\sum_{d\leqslant N}\frac{\varphi(d)}{d}x_d^2,\quad\hbox{with}\quad x_d:=\sum_{m\leqslant N/d}\frac{c_{md}}{\sqrt{m}}\cdot
\end{equation}
Let $\gamma\in]0,1[$ be an absolute constant. For all $y,z\in]\gamma,1]$, we invoke \eqref{loc} to bound $c_{md}$ from above by
$$\begin{cases}
\dsp y^{\Omega(md)-k}z^{\Omega(md,d)-\beta\log_2(2d)-C}\leqslant \frac{\gamma^{-C}y^{\Omega(md)}z^{\Omega(md,d)}}{(\log N)^{\beta\log y}(\log 2d)^{\beta\log z}}& \text{ if } d\leqslant \sqrt{N},\\
\dsp y^{\Omega(md)-k}z^{\Omega(md,N/d)-\beta\log_2(2N/d)-C}\leqslant \frac{\gamma^{-C}y^{\Omega(md)}z^{\Omega(md,N/d)}}{(\log N)^{\beta\log y}(\log 2N/d)^{\beta\log z}}& \text{ if } \sqrt{N}< d\leqslant N .
\end{cases}$$
Thus  
$$x_d\leqslant \frac{U(d,N)}{\gamma^{C}(\log N)^{\beta\log y}}\qquad (1\leqslant d\leqslant N),$$
with 
$$U(d,N):=
\begin{cases}
\dsp
\sum_{m\leqslant N/d}\frac{y^{\Omega(md)}z^{\Omega(md,d)}}{\sqrt{m} (\log 2d)^{\beta\log z}}& \text{ if } d\leqslant \sqrt{N},\\
\dsp\sum_{m\leqslant N/d}\frac{y^{\Omega(md)}z^{\Omega(md,N/d)}}{\sqrt{m}(\log 2N/d)^{\beta\log z}} & \text{ if } \sqrt{N}<d\leqslant N.
\end{cases}
$$
Evaluating the $m$-sums by means of standard estimates such as \cite[th. III.3.7]{Tencourse}, we get
$$x_d\ll\sqrt{\frac{N}{d}}\frac{y^{\Omega(d)}M(d,N)}{(\log N)^{\beta\log y}}$$
with 
$$M(d,N)\ll
\begin{cases}
z^{\Omega(d)}(\log 2d)^{yz-y-\beta\log z}(\log N)^{y-1}& \text{ if }d\leqslant \sqrt{N},\\
\noalign{\vskip-3mm}\\
 z^{\Omega(d,N/d)}(\log N/d)^{yz-1-\beta\log z}& \text{ if } \sqrt{N}<d\leqslant N.
\end{cases}
$$
It follows that
\begin{equation}
\label{majT}
\T(\bfc;N)\ll N\big\{S_1+S_2\big\}
\end{equation}
with
\begin{equation*}
\begin{aligned}
S_1&:=(\log N)^{2(y-1-\beta\log y)}\sum_{d\leqslant \sqrt{N}}\frac{(yz)^{2\Omega(d)}(\log 2d)^{2(yz-y-\beta\log z)}}{ d}\ll(\log N)^{\lambda}\\
S_2&:=(\log N)^{-2\beta\log y}\sum_{\sqrt{N}<d\leqslant N}\frac{y^{2\Omega(d)}z^{2\Omega(d,N/d)}}{ d}\Big(\log \frac{2N}{ d}\Big)^{2(yz-1-\beta\log z)}\\
&\ll(\log N)^\mu,
\end{aligned}
\end{equation*}
where we have put
\begin{equation*}
\begin{aligned}
&\lambda:=2(y-1-\beta\log y)+\big(y^2z^2+2yz-2y-2\beta\log z\big)^+,\cr 
&\mu:=y^2z^2-1-2\beta\log y+(2yz-1-2\beta\log z)^+,
\end{aligned}
\end{equation*}
provided the two quantities taken in positive value do not vanish.
Selecting $y=\beta$, $z=y(\beta)/\beta=1/\{1+y(\beta)\}$, we get
\begin{equation*}
\begin{aligned}
\lambda&=-2Q(\beta)+y(\beta)^2+2y(\beta)-2\beta+2\beta\log\{1+y(\beta)\}=f(\beta)-2Q(\beta)\\
\mu&=y(\beta)^2-2-2\beta\log \beta+2y(\beta)+2\beta\log \{1+y(\beta)\}\cr&=2\beta-y(\beta)^2-2-2\beta\log \beta+2\beta\log \{1+y(\beta)\}=f(\beta)-2Q(\beta)=\lambda.
\end{aligned}
\end{equation*}
Inserting back into \eqref{majT} and taking \eqref{c1} into account, we obtain the stated estimate.\par 
\smallskip
\begin{remark} Due to the multiplicative nature  of the definition of $\T_N$, deriving the lower bound is simpler and more  natural than for $\V_N.$ Let us provide some details in the perspective that a direct, possibly elementary approach  enables one to bound $\T_N/\V_N$ from above efficiently.\par 
Let $\beta,\,y\in]0,1]$ and $\bfc\in(\R^+)^N$ be such that $\norm\bfc_1=1$.  With $x_d$ as defined in \eqref{xd}, we have
$$S:=\sum_{d\leqslant N}y^{\Omega(d)}\frac{x_d}{\sqrt{d}}=\sum_{n\leqslant N}\frac{c_n}{\sqrt{n}}\sum_{d|n}y^{\Omega(d)}\geqslant \sum_{n\leqslant N}\frac{c_n(1+y)^{\omega(n)}}{\sqrt{n}}.$$ 
If \eqref{Dgrand} holds, then $S\gg(\log N)^{\beta\log (1+y)}/\sqrt{N}$, and so
$$\frac{(\log N)^{2\beta\log (1+y)}}{N}\ll S^2\ll \sum_{d\leqslant N}\frac{\varphi(d)}{d}x_d^2\sum_{d\leqslant N}\frac{y^{2\Omega(d)}}{\varphi(d)}\ll \T(\bfc;N)(\log N)^{y^2},$$
from which, as previously,
$$N\T(\bfc;N)\gg(\log N)^{f(\beta)}.$$\par 
If now we assume that \eqref{Dgrand} does not hold, then  
$$\sum_{d\in\D^-(\beta)}x_d=\sum_{n\leqslant N}\frac{c_n}{\sqrt{n}}\sum_{d|n,\,d\in\D^-(\beta)}\sqrt{d}\geqslant \sum_{n\in\D^-(\beta)}c_n\gg1.$$
Therefore
$$1\ll\sum_{d\in\D^-(\beta)}\frac{d}{\varphi(d)}\sum_{d\leqslant N}\frac{\varphi(d)}{d}x_d^2\ll \frac{\T(\bfc;N)N}{(\log N)^{Q(\beta)}\sqrt{\log_2N}}.$$
\par 
In view of the definition of $\eta$, we get in all circumstances
$\T_N\gg{(\log N)^\eta}.$
\end{remark}

\section{Proof of Theorem \ref{multh}}
The lower bound immediately follows from the Cauchy-Schwarz inequality. Indeed, defining
 $$r(n):=\sum_{\substack{dt=n\\ d,t\leqslant N}}c_dc_t,$$
 we have, under the assumption $\norm\bfc_1=1$, that
$$1=\bigg(\sum_{n\leqslant N^2}r(n)\bigg)^2\leqslant H(N)\sum_{n\leqslant N^2}r(n)^2=H(N)\E(\bfc;N)$$
where $H(N)$ is the number of those integers not exceeding $ N^2$ that are products of two integers $\leqslant N$. By
 \cite[cor.~3]{Ford}, we have $$H(N)\ll N^2/\big\{(\log N)^\delta(\log_2N)^{3/2}\big\}.$$ The lower bound contained in~\eqref{estEN} follows.\par \goodbreak
To establish the upper bound, select $n\mapsto c_n$ as the indicator function of the set of those integers $n\in]N/2,N]$ satisfying \eqref{loc} with now $\kappa:=1/\log 4$, so that, by the corresponding version of \eqref{c1},
\begin{equation}
\label{vmr}
\norm\bfc_1^2=\sum_{m\geqslant 1}r(m)\gg\frac{N^2}{ (\log N)^\delta(\log_2N)^3}.
\end{equation}
Moreover, for any integer $m\leqslant N^2$, we have $r(m)\leqslant \sum_{d|m, N/2<d\leqslant N}1$. Observe that $r(m)>0$ implies that $m$ satisfies \eqref{loc} with $\kappa=1/\log 2$ and suitable $C$. We may therefore apply inequality (2.46) from \cite{HT88}, which yields, since $Q(\kappa)=\delta$,
$$\sum_{n\leqslant N^2}r(n)^2\ll\frac{N^2}{ (\log N)^\delta},$$
provided $O(1)$ is substituted for $L(\log v)$  in the proof, which induces no perturbation. 
Considering \eqref{vmr}, we get the required estimate.

\section{Logarithmic improvement of Burgess' bound}

\subsection{Preliminary results}
\hspace{\parindent} 
The following result is a consequence of the Weil bounds for complete character sums, see for instance~\cite[Lemma~12.8]{IK}.
\begin{lemma}
\label{moments}
Let $r \geqslant 2$ be an integer, $B\geqslant 1$,  $p$ a prime and let $\chi$ a non principal Dirichlet character  modulo $p$. We have
\begin{equation}
\label{majmom}
\sum_{1\leqslant \ell\leqslant p}\Vbs{\sum_{1\leqslant b \leqslant B}\chi(\ell+b)}^{2r}\leqslant  (2r)^r B^{r}p + 2rB^{2r}\sqrt{p}.
\end{equation}
\end{lemma}

For $A\geqslant 1$, $\bfc\in\R^A$, $M\geqslant 1$, $N\geqslant 1$, $\ell\geqslant 1$,  define 
\begin{equation}
\label{defrT}
r(\ell;\bfc):=\sum_{M<m\leqslant M+N} \sum_{a\ell\md mp}c_a,\quad R(\bfc;A,M,N):=\sum_{\ell\geqslant 1}r(\ell;\bfc)^2.
\end{equation}  
The following easy lemma provides an upper bound for the latter quantity.

\begin{lemma}\label{abcd}
For prime $p$, integers $A$, $M$, $N$,  such that 
\begin{equation}\label{bornes} A \leqslant N, \hspace{3mm} AN \leqslant p,  \end{equation}  
and all $\bfc\in(\R^+)^N$, we have 
\begin{equation}
\label{majRc}
 R(\bfc;A,M,N) \leqslant  \norm\bfc_1^2+2N\V(\bfc;A)  .
 \end{equation}
 \end{lemma}
\begin{proof} 
Plainly
$$R(\bfc;A,M,N)=\sum_{a,b\leqslant A}c_ac_b\sum_{\substack{M<m,n\leqslant M+N\\ bm\md{an}p}}1.$$
Let $d:=(a,b)$, $a_1:=a/d$, $b_1:=b/d$. Since $AN\leqslant p$, there is at most one multiple of $p$ in the interval $]-AN,AN[$ and so the variables $m,n$ in the inner sum satisfy $m=m_0+ta_1$, $n=n_0+tb_1$ for some integer $t$ and fixed $m_0$, $n_0$ depending on $a$, $b$ and $p$. The condition $m,n\in]M,M+N]$ then implies that there are at most $$ 1+N/\max(a_1,b_1)\leqslant 1+2N(a,b)/(a+b)$$ choices for $t$. This is all we need. 
\end{proof}
\subsection{Proof of Theorem \ref{Burgess}}
 \hspace{\parindent} We retain the notation from \cite{Burgessrefine} and follow closely the argument.   We set
$$ \U_p:=\max_{1\leqslant x\leqslant p} \V_x$$ and proceed by induction on $N$. The induction hypothesis is that
there exists some constant $C$ such that
\begin{equation}
\label{induction}
\Vbs{\sum_{M<n\leqslant M+K}\chi(n)}\leqslant C K^{1-1/r}p^{(r+1)/4r^2} \U_p^{1/2 r}\quad(M\geqslant 1,\,1\leqslant K<N),
\end{equation} 
and set out to prove that this persists to hold for $N+1$.
\par \goodbreak
As in \cite{Burgessrefine}, $N<p^{1/4}$ forms the basis of our induction since \eqref{induction} trivially holds in this range. Still following \cite{Burgessrefine}, we define 
\begin{equation*}
A:=\fl{\frac{N}{16rp^{1/2r}}},\qquad   B:=\fl{rp^{1/2r}},
\end{equation*}
and note that, for any integers $a$, $b$, with $1\leqslant a\leqslant A$, $1\leqslant b\leqslant B$, we have 
\begin{align*}
\sum_{M<n\leqslant M+N}\chi(n)&= \sum_{M<n\leqslant M+N}\chi(n+ab) +V-W,
\end{align*}
with
$$V:=\sum_{M-ab<n\leqslant M}\chi(n+ab),\quad W:=\sum_{M+N-ab<n\leqslant M+N}\chi(n+ab).$$ 
By the induction hypothesis,  
$$
\max\big(|V|,|W|\big)\leqslant \tfrac{1}{4}C N^{1-1/r}p^{(r+1)/4r^2}   \U_p^{1/2r},
$$
which combined with the above implies that 
$$
\Vbs{\sum_{M<n\leqslant M+N}\chi(n)-\sum_{M<n\leqslant M+N}\chi(n+ab) }\leqslant \tfrac{1}{2}C N^{1-1/r}p^{(r+1)/4r^2}  \U_p^{1/2r}.
$$
 \par 
The main difference with the method of \cite{Burgessrefine} comes from our choice of the subset used to average. Given $\bfc\in(\R^+)^A$, we multiply the last inequality by $c_a$  and sum over $1\leqslant a\leqslant A$, $1\leqslant b \leqslant B$, to obtain 
\begin{equation}
\label{eq:Win}
\Vbs{\sum_{M<n\leqslant M+N}\chi(n)}\leqslant \frac{S}{B\norm\bfc_1}+\tfrac{1}{2}C N^{1-1/r}p^{(r+1)/4r^2} \U_p^{1/2r},
\end{equation}
with 
\begin{equation}
\label{eq:WW123}
S:=\sum_{M<n\leqslant M+N}\sum_{1\leqslant a \leqslant A}c_a\Vbs{\sum_{1\leqslant b \leqslant B}\chi(n+ab)}.
\end{equation}
Multiplying the innermost summation in (\ref{eq:WW123}) by $\overline{\chi(a)}$ and 
collecting the values of $n/a\, (\bmod\, p)$, we arrive at 
\begin{equation}
\label{eq}
S= \sum_{1\leqslant \ell\leqslant p}r(\ell;\bfc) \Vbs{\sum_{1\leqslant b \leqslant B}\chi(\ell+b)}.
\end{equation}
Along the lines of \cite{Burgessrefine}, we apply H\"older's inequality to get, for $r\geqslant 2$, 
\begin{equation}
\label{holder}
S^{2r}\leqslant \bigg(\sum_{1\leqslant \ell\leqslant p}r(\ell;\bfc) \bigg)^{2r-2}\bigg(\sum_{1\leqslant \ell\leqslant p}r(\ell;\bfc)^2 \bigg)\bigg(\sum_{1\leqslant \ell\leqslant p}\Vbs{\sum_{1\leqslant b \leqslant B}\chi(u+b)}^{2r}\bigg).
\end{equation}
We trivially have
$$\sum_{1\leqslant \ell\leqslant p}r(\ell;\bfc)=N\norm\bfc_1.$$
Applying \eqref{majRc} to the second factor on the right-hand side of \eqref{holder} and \eqref{majmom} to the third, we obtain
$$S^{2r}\ll N^{2r-2}\norm\bfc_1^{2r-2}\Big\{\norm\bfc_1^2+N\V(\bfc;A)\Big\}\Big\{B^rp+B^{2r}\sqrt{p}\Big\}. $$
Select $\bfc$ such that $\V(\bfc;A)=\V_A\leqslant \U_p$. Hence $N\V(\bfc;A)\leqslant  N\U_p\norm\bfc_1^2/A$. This  upper bound being $\gg\norm\bfc_1^2$ and the last expression between curly brackets being $\ll (2r)^{2r-1}p^{3/2}$, we infer that
$$S^{2r}\ll N^{2r-1}\norm\bfc_1^{2r}\U_pp^{3/2}/A\ll N^{2r-2}\norm\bfc_1^{2r}(2r)^{2r-1}p^{(3r+1)/2r}\U_p$$
and so, there exists an absolute constant $C_0$ such that
$$\frac{S}{B\norm\bfc_1}\leqslant C_0 N^{1-1/r}p^{(r+1)/4r^2}\U_p^{1/2r}.$$
Selecting $C\geqslant 2C_0$ in \eqref{eq:Win} terminates the proof.

\section{Proof of Theorem \ref{mainth}}\label{ProofTh}
\hspace{\parindent} 
Let $p$ be a prime, and put $q:=\lfloor\sqrt{p/3}\rfloor$. For  even character $\chi\in X_p^+$ and $\bfc\in(\R^+)^q$, we consider 
 \begin{equation}\label{weight} M(\chi)=\sum_{m\leqslant q }c_m\overline{\chi(m)},\end{equation}   and the first  mollified moment 
\begin{equation}\label{mollif}
M_1(p) :=\sum_{\chi\in X_p^+}M(\chi)\vartheta(x;\chi) .
\end{equation} 
We further define 
\begin{equation*}
\begin{aligned}
M_0(p)&:=\big|\big\{ \chi\in X_p^+:\vartheta(x;\chi)\neq 0 \big\}\big|,\\ 
M_2(p)&:= \sum_{\chi\in X_p^+}\vert \vartheta(x;\chi)\vert^2,\  M_4(p):=\sum_{\chi\in X_p^+}\vert M(\chi)\vert^4.
\end{aligned}
\end{equation*}
 By H\"{o}lder's inequality, we have 
\begin{equation}\label{Holder} M_1(p) \leqslant M_2(p)^{1/2} M_4(p)^{1/4} M_0(p)^{1/4}.  \end{equation}
\par 
 In \cite{LM}, the authors achieved an asymptotic formula for the fourth moment of $\vert \vartheta(x;\chi)\vert$ showing that, after expanding the fourth power, the main contribution comes from the solutions to $mn = m'n'$. They actually obtained a precise asymptotic formula for the related counting function:
$$\big|\big\{m,m',n,n':mn=m'n', m^2+n^2+m'^2+n'^2 \leqslant x\big\}\big| \sim \tfrac{3}{ 8}x\log x.$$ \par 
In order to improve on this result, we need to reduce the effect of the logarithmic factor. Using \eqref{Holder}, we can relate the problem to a similar one involving the weights~$c_n$. We recall the notation $q:=\lfloor\sqrt{p/3}\rfloor$ from above and $\E(\bfc;N)$ from \eqref{weightedenergy}.
\par 
\begin{lemma}\label{moment1} For large prime $p$ and any sequence $\bfc\in (\R^+)^q$, we have
\begin{equation} 
\label{minM0}
M_0(p)  \gg \norm\bfc_1^4 /\E(\bfc;q).
\end{equation} 
In particular,  $M_0(p) \gg p/\E_q.  $
\end{lemma} 

\begin{proof}
Recall the classical orthogonality relations for the subgroup $X_p^+$  of even Dirichlet  characters modulo $p$, viz.
$$\sum_{\chi\in X_p^+}\chi (m)\overline{\chi (n)}
=\begin{cases}\tfrac12
(p-1) & \textrm{if } m\md{\pm n}p \textrm{ and } p\nmid m,\\
0 & \textrm{otherwise}. \end{cases}$$
Thus   
\begin{equation}\label{lowM1}
\begin{aligned}
M_1(p) &=  \sum_{\chi\in X_p^+} \sum_{m\leqslant q} c_m\overline{\chi (m)}\sum_{n\geqslant 1}\chi(n)\e^{-\pi n^2x/p}\\& \geqslant  \tfrac12(p-1)\sum_{m\leqslant q}c_m \e^{-\pi m^2x/p}\gg p\norm\bfc_1.
\end{aligned} 
\end{equation}
\par 
Similarly, with the notation \eqref{weightedenergy},
\begin{equation}
\label{M4p}
M_4(p)=
\sum_{ m,m', n,n'  \leqslant q }c_mc_{m'}c_nc_{n'}
\sum_{\chi\in X_p^+}\chi (m'n')\overline{\chi (mn)}
=\tfrac12(p-1)\E(\bfc;q)
\end{equation}

and 
$$
M_2(p)= \tfrac12
(p-1)  \sum_{ \substack{ m,n\geqslant 1\\
m\md{\pm n} p }}      \e^{- \pi (m^2+n^2)x/p} \ll p^{3/2}.
$$
Inserting these estimates back into \eqref{Holder} yields \eqref{minM0}.
\end{proof}  

Combining Lemma \ref{moment1} with Theorem \ref{multh} completes  the proof of Theorem \ref{mainth}.

\section{Proof of Theorem \ref{firstmomentcarac}}
\hspace{\parindent} We adopt  techniques similar to those used in Section \ref{ProofTh}. Parallel to \eqref{weight}, given $\bfc\in(\R^+)^N$, we define $$M(N;\chi)=\sum_{m\leqslant N}c_m\overline{\chi(m)}.$$ 
For $r\in]0,4/3[$, let $ u:=r/(4-2r)$, $ v:=(4-3r)/(4-2r)$ so that $u+v=1$. Further define $s:=4-2r$ and $t:=(8-4r)/(4-3r)$ and note that $1/s+1/t+1/4=1$.
Finally, let us put, for $p>2$,
$$\mathfrak{S}_k(N):= \frac{1}{p-2} \sum_{\chi \neq \chi_0} \vert S(N;\chi) \vert^k\quad(k>0), \quad \mathfrak{M}_4(N):=\frac{1}{p-2} \sum_{\chi \neq \chi_0} \vert M(N;\chi) \vert^4. $$
\par  
Let us first assume $N\leqslant \sqrt{p}$. 
Representing $S(N;\chi)=S(N;\chi)^{ u}S(N;\chi)^{ v}$ and applying H\"{o}lder's inequality, we get 
\begin{equation}
\label{momentsr} 
\frac{1}{p-2}\Vbs{\sum_{\chi \neq \chi_0}S(N;\chi)M(N;\chi)} \leqslant \mathfrak{S}_r(N)^{1/s} \mathfrak{S}_2(N)^{1/t} \mathfrak{M}_4(N)^{1/4},  \end{equation}
\par 
Orthogonality relations, readily yield that $\mathfrak{S}_2(N) \ll N$. Similarly to \eqref{lowM1}, the left-hand side of \eqref{momentsr} is $\gg \norm\bfc_1$, and $\mathfrak{M}_4(N) \ll \E(\bfc;N)$ as in \eqref{M4p}. Combining  these bounds, we deduce
$$\mathfrak{S}_r(N)\gg\norm\bfc_1^s\E(\bfc;N)^{-s/4}N^{-s/t}\gg N^{r/2}/\E_N^{1-r/2},$$
by choosing $\bfc$ optimally.
\par 
In order to extend the range to $\sqrt{p}<N\leqslant \tfrac12p$, we appeal to Plya's formula
\begin{equation}
\label{Polya}
\sum_{n\leqslant p/N}\chi(n)=\frac{\tau(\chi)}{2\pi i}\sum_{0<|h|\leqslant H}\frac{\overline{\chi(h)}}h\Big(1-\e^{-2\pi ih/N}\Big)+O\Big(1+\frac{p\log p}H\Big)\qquad (H\geqslant 1),
\end{equation}
where $\tau(\chi)$ is the usual Gauss sum of modulus $\sqrt{p}$. We select $H:=\sqrt{Np}(\log p)^2$. Restricting to odd characters, we may replace the complex exponentials by cosines and argue as in the first part with the $h$-sum.
 
\section{Concluding remarks}
\label{combiproblems}
\hspace{\parindent} 
 Under the additional restriction $\bfc\in \{0,1\}^N$, the first problem considered in Section \ref{Intro} is equivalent to  constructing  a set $\B \subset [1,N]$ of high density such that the associated Gl sum is small. In Section \ref{gcdsumth}, we showed that a large subset $\B$ of $\{n\leqslant N:\Omega(n)=\fl{\beta\log_2N}\}$ with $\beta \approx 0.48155$  verifies 
$$\V(\1_\B;N) \ll \vert \B \vert (\log\vert \B \vert)^{o(1)}$$ or, in another words, that the multiplicative energy verifies
$$E([1,N],\B) \ll N \vert \B \vert (\log N)^{o(1)}. $$ Theorem \ref{gcdsumth} tells us that this is essentially the densest set with this property. \par 
In the symmetric case, the question becomes: find the maximal $\beta=\beta_N\in]0,1]$ such that there exists a set $\B \subset \left[1,N\right]$ of density $\beta$ satisfying $ E(\B ,\B )\ll \vert \B \vert^2 (\log N)^{o(1)}. $ In Theorem~\ref{multh}, we proved that a large subset of $\{n\leqslant N:\Omega(n)=\fl{(\log_2N)/\log 4}\}$ yields the optimal density $(\log N)^{-\delta/2+o(1)}$.

\section*{Acknowledgements}\hspace{\parindent} 
The authors express their gratitude to Kannan Soundararajan for drawing their attention to~\cite{BS16}, to Kevin Ford for pointing to his work \cite{Fo07}, thereby improving the upper bounds in Theorems \ref{gcdsumth} and \ref{multh}, and to Ping Xi for the observations quoted after the statement of Theorem \ref{firstmomentcarac}. The second author thanks St\'{e}phane Louboutin for valuables remarks, and Igor Shparlinski for indicating  reference~\cite{Burgessrefine} after a first version of the draft was released. 
 second author also acknowledges support of the Austrian Science Fund (FWF), stand-alone project P 33043  ``Character sums, L- functions and applications'' and START-project Y-901 ``Probabilistic methods in analysis and number theory'', headed by Christoph Aistleitner.

\bibliographystyle{alpha}

%\bibliography{theta}
\Addresses

\end{document}